\documentclass{amsart}
\usepackage{amsmath,amssymb,enumerate,etoolbox,mathrsfs}
\newtheorem{theorem}{Theorem}
\newtheorem*{theorem*}{Theorem}

\newtheorem*{conjecture*}{Conjecture}

\newtheorem*{proposition*}{Proposition}
\newtheorem{lemma}{Lemma}
\newtheorem*{lemma*}{Lemma}
\newtheorem{lemmaa}{Lemma}

\theoremstyle{definition}

\newtheorem*{remark*}{Remark}

\newcommand{\F}{\mathbb{F}}

\newcommand{\Q}{\mathbb{Q}}
\newcommand{\Z}{\mathbb{Z}}
\newcommand{\SL}{{\rm SL}}
\newcommand{\GL}{{\rm GL}}
\newcommand{\Gal}{{\rm Gal}}
\newcommand{\Fact}{{\rm Fact}\,}
\newcommand{\ord}{{\rm ord}}
\newcommand{\supp}{{\rm supp}\,}
\newcommand{\Irr}{{\rm Irr}}

\newcommand{\Cl}{{\rm Cl}}
\newcommand{\cF}{\mathcal{F}}
\newcommand{\cG}{\mathcal{G}}
\newcommand{\cP}{\mathcal{P}}
\newcommand{\cR}{\mathcal{R}}
\newcommand{\cX}{\mathcal{X}}
\newcommand{\Prob}{\mathbf{P}}

\title[Many zeros of many characters of ${\rm GL}(n,q)$]{Many zeros of many characters of 
$\boldsymbol{\GL(n,q)}$}
\author[P.~X.~Gallagher]{Patrick X.~Gallagher}
\address{Department of Mathematics, Columbia University, New York, NY, USA}
\email{pxg@math.columbia.edu}
\author[M.~J.~Larsen]{Michael J.~Larsen}
\address{Department of Mathematics, Indiana University, Bloomington, IN, USA}
\email{mjlarsen@indiana.edu}
\author[A.~R.~Miller]{Alexander~R.~Miller}
\address{Fakult\"at f\"ur Mathematik, Universit\"at Wien, Vienna, Austria}
\email{alexander.r.miller@univie.ac.at}
\begin{document}
\begin{abstract}
For $G=\GL(n,q)$, the proportion $P_{n,q}$ of pairs $(\chi,g)$ in $\Irr(G)\times G$ 
with $\chi(g)\neq 0$ satisfies $P_{n,q}\to 0$ as $n\to\infty$. 
\end{abstract}
\thanks{ML was partially supported by the NSF grant DMS-1702152.}
\maketitle
\thispagestyle{empty}
\section{Introduction}
A few years ago, it was shown \cite{ARMiller1} that for 
$G=S_n$ the proportion $P_n$
of pairs $(\chi,g)$ in $\Irr(G)\times G$ 
with $\chi(g)\neq 0$ satisfies 
\begin{equation}\label{Sn:Limit}
P_n\to 0\ \text{as}\ n\to\infty.
\end{equation}
Here we prove the analogous statement for $\GL(n,q)$:
\begin{theorem}\label{Theorem:GL(n,q):Bound}
The proportion $P_{n,q}$, in $\Irr({\rm GL}(n,q))\times {\rm GL}(n,q)$, 
of pairs $(\chi,g)$ with $\chi(g)\neq 0$ satisfies 
\begin{equation}\label{Thm:1:Limit}
\sup_q P_{n,q}\to 0\ \text{as}\ n\to\infty.
\end{equation}
\end{theorem}

One of the two proofs of \eqref{Sn:Limit} in \cite{ARMiller1} 
is based on the special property of $S_n$, derived from 
estimates due to Erd\H{o}s and Lehner \cite{ErdosLehner} 
and Goncharoff \cite{Goncharoff}, that for large $n$, 
a suitably chosen small proportion of $\Cl(S_n)$ covers all 
but a small proportion of $S_n$. For the proof of 
\eqref{Thm:1:Limit} for $\GL(n,q)$, we use both conjugacy class 
sizes and character degrees. There is a general inequality, \eqref{Lemma:General:Bound} below, 
proved in Section~\ref{Sect:Proof:Lemma:A}, 
and special properties \eqref{Lemma:B:first}, \eqref{Lemma:B:second} 
of the degrees and sizes 
of almost all characters and classes of $\GL(n,q)$, which are 
proved in Section~\ref{Sect:Proof:Lemma:B}.

To lighten the notation, for a finite group $G$ we denote by 
$d_\chi$ the degree $\chi(1)$ of an (irreducible) character 
$\chi$ of $G$, by $s_g$ the size $|g^G|$ of the conjugacy class 
$g^G$, and by $(d_\chi,s_g)$ the greatest common divisor of $d_\chi$ and $s_g$.

\begin{lemmaa}\label{Lemma:A} 
For each finite group $G$ and $\varepsilon >0$, 
the proportion $P$, in $\Irr(G)\times G$, of pairs
$(\chi,g)$ with $\chi(g)\neq 0$ satisfies 
\begin{equation}\label{Lemma:General:Bound}
P\leq Q(\varepsilon)+\varepsilon^2,
\end{equation}
with $Q(\varepsilon)$ the proportion, in $\Irr(G)\times G$, of 
pairs $(\chi,g)$ with $(d_\chi,s_g)/d_\chi\geq \varepsilon$.
\end{lemmaa}

\begin{lemmaa}\label{Lemma:B}
For all $\delta,\varepsilon>0$, there exists $N$ such that if $n\ge N$, $q$ is a prime power, and $G = \GL(n,q)$, then
for $(\chi,g)$ in $\Irr(G)\times G$,
\begin{equation}\label{Lemma:B:first}
\frac{(d_\chi,s_g)}{d_\chi}<\varepsilon,
\end{equation}
except for $(\chi,g)$ in a subset 
$\cR\subset \Irr(G)\times G$ 
such that 
\begin{equation}\label{Lemma:B:second}
|\cR| \leq \delta |\Irr(G)\times G|.
\end{equation}
\end{lemmaa}

\section{Proof of Theorem~\ref{Theorem:GL(n,q):Bound} using 
Lemmas~\ref{Lemma:A} and \ref{Lemma:B}}

For $G=\GL(n,q)$ and $\varepsilon>0$, Lemma~\ref{Lemma:A} 
gives 
\[P_{n,q}\leq Q_{n,q}+\varepsilon^2,\]
with $P_{n,q}$ the proportion of pairs $(\chi,g)$ with $\chi(g) \neq 0$ and $Q_{n,q}$ the proportion of pairs
with $(d_\chi,s_g)/d_\chi\geq\varepsilon$. 
Lemma~\ref{Lemma:B} gives
$Q_{n,q}\leq \delta$
for $n\ge N$.
Thus for $n$ sufficiently large,
\[P_{n,q}\le \delta+\varepsilon^2,\]
from which Theorem~\ref{Theorem:GL(n,q):Bound} follows.\qed

\section{Proof of Lemma~\ref{Lemma:A} by a device of Burnside}\label{Sect:Proof:Lemma:A}
For each $\chi\in \Irr(G)$ and 
$g\in G$, both $\chi(g)$ and $s_g\chi(g)/d_\chi$ are algebraic integers, 
 so for all $a,b\in\mathbb Z$, so is $(ad_\chi+b s_g)\chi(g)/d_\chi$. 
Choosing $a$ and $b$ so that $ad_\chi+bs_g$ is the greatest common divisor 
$(d_\chi,s_g)$ of $d_\chi$ and $s_g$, this gives 
\begin{equation}\label{chi:alpha:eq}
\chi(g)=\frac{d_\chi}{(d_\chi,s_g)}\alpha_{\chi,g},
\end{equation}
with $\alpha_{\chi,g}$ an algebraic integer in the cyclotomic 
field $\Q(\zeta_{|G|})$ with $\zeta_{|G|}=e^{2\pi i/|G|}$. 

From \eqref{chi:alpha:eq}, for each $\chi$,
\begin{equation}\label{G:Sum}
\sum_{g\in G}\bigl(\frac{d_\chi}{(d_\chi,s_g)}\bigr)^2|\alpha_{\chi,g}|^2=|G|.
\end{equation}
To \eqref{G:Sum}, apply elements $\sigma$ of the 
Galois group $\Gamma = \Gal(\mathbb \Q(\zeta_{|G|})/\Q)$, average over 
$\Gamma$, and use the fact, due to Burnside, that the average 
over $\Gamma$ of $|\sigma(\alpha)|^2$ is $\geq 1$ for each 
non-zero algebraic integer $\alpha\in\Q(\zeta_{|G|})$, \cite[p.~359]{PXGallagher2}. 
This gives, for each $\chi$, 
\begin{equation}\label{G:Sum:Ineq:1}
{\sum_{g\in G}}'\bigl(\frac{d_\chi}{(d_\chi,s_g)}\bigr)^2\leq |G|,
\end{equation}
the dash meaning that the sum is over those $g$ with $\chi(g)\neq 0$. 
From \eqref{G:Sum:Ineq:1},
\begin{equation}\label{G:Sum:Ineq:2}
\sum_{\chi\in\Irr(G)}{\sum_{g\in G}}'(\frac{d_\chi}{(d_\chi,s_g)})^2\leq |\Irr(G)||G|.
\end{equation}
From \eqref{G:Sum:Ineq:2}, the proportion, 
in $\Irr(G)\times G$, of pairs $(\chi,g)$ with both 
$\chi(g)\neq 0$ and $(d_\chi,s_g)/d_\chi\leq \varepsilon$ is at most 
$\varepsilon^2$, from which \eqref{Lemma:General:Bound} follows. 
\qed

\section{Number theoretic lemmas: partitions}

We denote by $p(n)$ the number of partitions of a non-negative integer $n$.

\begin{lemma}
\label{power2}
For each positive integer $n$, $p(n) \le 2^{n-1}$.
\end{lemma}

\begin{proof}
The base case $n=1$ is trivial.  For $n>1$,
the number of partitions with smallest part $m$ is at most $p(n-m)$, so
\[p(n) \le 1+p(1) + p(2) +\cdots + p(n-1) \le 1+1+2 +\cdots+2^{n-2} = 2^{n-1},\]
and the lemma follows by induction.
\end{proof}

\begin{lemma}
\label{phi}
Let $\phi := \frac{1+\sqrt 5}2$.  Then $p(n) \le \phi^n$ for all non-negative integers $n$.
\end{lemma}

\begin{proof}
The partition function is non-decreasing since the number of partitions of $n+1$ with a part of size $1$ is $p(n)$.
The lemma holds for $n\in \{0,1\}$.  For $n\ge 2$, the pentagonal number theorem implies
\begin{equation}
\label{recur}
p(n) = p(n-1) + p(n-2) - p(n-5) - p(n-7) + p(n-12) + \cdots,
\end{equation}
with sign pattern $++--++--++--\cdots$ and where the sum on the right-hand side terminates at the last term $\pm p(n-m)$, where $m$ is the largest generalized pentagonal number for which $n\ge m$.
By monotonicity, the right-hand side of \eqref{recur} is at most $p(n-1)+p(n-2)$, so the lemma follows by induction on $n$.
\end{proof}

\begin{lemma}
\label{rare}
There exists $\gamma<1$ such that if $q\ge 2$ and $a$ and $b$ are positive integers such that $a(b-1) \ge N\geq 0$, then
\[\frac{p(b)}{q^{a(b-1)}} < 2\gamma^N.\]
\end{lemma}

\begin{proof}
It suffices to prove the lemma for $q=2$.  For $a=1$, we have $b-1\ge N$, so Lemma~\ref{phi} implies
\[\frac{p(b)}{2^{a(b-1)}} = \frac{p(b)}{2^{b-1}} < 2(\phi/2)^N.\]
For $a\ge 2$, $a(b-1) \le 2(a-1)(b-1)$, so by Lemma~\ref{power2},
\[\frac{p(b)}{2^{a(b-1)}} \le 2^{-(a-1)(b-1)} \le  (1/\sqrt2)^N < 2(1/\sqrt2)^N.\]
Therefore, we may take $\gamma = \phi/2 > 1/\sqrt 2$.
\end{proof}
\section{Number theoretic lemmas: cyclotomic polynomials}

For $n$ a positive integer, let $\Phi_n(x)$ denote the minimal polynomial over $\Q$ of~$e^{2\pi i/n}$.  Thus
\begin{equation}
\label{standard}
x^n-1 = \prod_{d\vert n} \Phi_d(x),
\end{equation}
so by M\"obius inversion,
\begin{equation}
\label{invert}
\Phi_n(x) = \prod_{d\vert n} (x^{n/d}-1)^{\mu(d)}.
\end{equation}

For any prime $\ell$, let $\ord_\ell(x)$ denote the largest integer $e$ such that $\ell^e$ divides $x$.

\begin{lemma}
\label{cong}
Let $\ell$ be a prime, $e$ a positive integer, and $n$ an integer such that  $\ord_\ell(n-1)=e$.
\begin{enumerate}[\rm(i)]
\item\label{cong:prime}
If $k$ is a positive integer prime to $\ell$, then $\ord_\ell(n^k-1) = e$.  
\item\label{cong:odd} If $\ell$ is odd and $\ord_\ell(k) = 1$, then $\ord_\ell(n^k-1) = e+1$.
\end{enumerate}
\end{lemma}

\begin{proof}
Let $n=1+m \ell^e $, where $\ell\nmid m$.  By the binomial theorem,
\[n^k\equiv 1 + km \ell^e \pmod{\ell^{2e}},\]
which implies claim \eqref{cong:prime}.  For claim \eqref{cong:odd}, using part \eqref{cong:prime}, it suffices to treat the case $k=\ell$, for which we have
\[n^\ell\equiv 1+ m \ell^{e+1} + \frac{m^2(\ell-1)}2 \ell^{2e+1} \pmod{\ell^{3e}}.\qedhere\]
\end{proof}

\begin{lemma}
\label{big-factor}
Suppose $n>0$ and $a>1$ are integers.  We factor $\Phi_n(a)$ as $P_n(a)R_n(a)$, where $P_n(a)$ is relatively prime to $n$ and $R_n(a)$ factors into prime divisors of $n$.  
\begin{enumerate}[\rm(i)]
\item\label{b-f:mod} Every prime divisor of $P_n(a)$ is $\equiv 1$ (mod $n$).
\item\label{b-f:s-f} If $n\ge 3$, $R_n(a)$ is a square-free divisor of $n$.
\item\label{b-f:bound} For $n\ge 3$, $P_n(a) > 2^{\sqrt{n/2}-\log_2 n - 2}$.
\item\label{b-f:ord} If $m\ell>n$ and $\ell$ is a prime divisor of $P_m(a)$, then
\[\ord_\ell (a^n-1) = \begin{cases}  \ord_\ell P_m(a)&\text{if $m\mid n$,} \\ 0&\text{otherwise.}\end{cases}\]
\end{enumerate}

\end{lemma}

\begin{proof}
Fix any prime $\ell$ which divides $\Phi_n(a)$.  As $\ell\mid a^n-1$, $a$ is not divisible by $\ell$, so it represents a class in $\F_\ell^\times$.
Let $k$ be the order of this class.  As $a^n\equiv 1\pmod\ell$, $k\mid n$. 
Let $s$ denote the largest square-free divisor of $n/k$.  By \eqref{invert},
\[\ord_\ell\Phi_n(a) = \ord_\ell \prod_{d\mid s} (a^{n/d}-1)^{\mu(d)}.\]

Now, if $s$ can be written $p s'$ for some prime $p\neq \ell$,
\begin{equation}
\label{cube}
\prod_{d\mid s} (a^{n/d}-1)^{\mu(d)} = \prod_{d\mid s'}  \Bigl(\frac{a^{n/d}-1}{a^{n/pd}-1}\Bigr)^{\mu(d)}.
\end{equation}
Applying part \eqref{cong:prime} of Lemma~\ref{cong} with $k=p$, the above formula implies $\ord_\ell \Phi_n(a) = 0$, contrary to assumption.
Since $s$ is square-free, it follows that it can only be $1$ or $\ell$.

If $\ell$ divides $P_n(a)$, then it does not divide $n$.  That means $s=1$, so the class of $a$ has order $n$ in a group of order $\ell-1$.  
This implies part \eqref{b-f:mod}.  Conversely, if $\ell$ does divide $n$, it cannot be $1$ (mod $n$), so $s=\ell$.

If $s=\ell>2$, then $d$ square-free and $\ord_\ell (a^{n/d}-1) > 0$ implies $d\in \{1,\ell\}$.
Therefore, part \eqref{cong:odd} of Lemma~\ref{cong} implies that the left-hand side of \eqref{cube} has $\ord_\ell$ equal to $1$.
If $s=\ell=2$, then $k=1$, so we need only consider the case that $n$ is a power of~$2$.  For $t\ge 2$, $\Phi_{2^t}(x) = (x^{2^{t-2}})^2+1$,
so plugging in $a$, the result has at most one factor of $2$.  This gives claim \eqref{b-f:s-f}.

By \eqref{invert}, 
\begin{equation}
\label{quarter}
\Phi_n(a) \ge a^{\deg\Phi_n}\prod_{i=1}^\infty (1-a^{-1}) \ge a^{\phi(n)} \prod_{i=1}^\infty (1-2^{-1}) \ge \frac{2^{\phi(n)}}4.
\end{equation}
As $\phi(p^e) \ge \sqrt{p^e}$ except when $p^e=2$, the multiplicativity of $\phi$ implies $\phi(n) \ge \sqrt{n/2}$.  By part \eqref{b-f:s-f}, $R_n(a) \le n$, and claim \eqref{b-f:bound} follows.

If $\ell$ divides $P_m(a)$, then the image of $a$ in $\F_\ell^\times$ is of order $m$, so $\ell$ divides $a^n-1$ only if $n$ is divisible by $m$.
In that case, $P_m(a)$ divides $\Phi_m(a)$, which is a divisor of $a^m-1$ and therefore $a^n-1$.  Moreover, $\ell$ does not divide $m$, so $\ord_\ell P_m(a) = \ord_\ell \Phi_m(a)$.
To prove \eqref{b-f:ord}, it remains to show that $a^n-1$ has no additional factors of $\ell$ beyond those in $a^m-1$.
It suffices to prove that $\Phi_{n'}(a)$ is not divisible by $\ell$ if $n'$  is a divisor of $n$ and $m$ is a proper divisor of $n'$.  Indeed, $\ell$ does not divide $P_{n'}(a)$
because $a$ is not of order exactly $m'$ (mod $\ell$).  If it divides $\Phi_{n'}(a)$, it must divide $R_{n'}(a)$, so it must divide $n'$.  It does not divide $m$, so it must divide $n'/m\le m$.
This is ruled out by \eqref{b-f:mod}.
\end{proof}

\section{Irreducible characters of $\GL(n,q)$}
In what follows, $G=\GL(n,q)$.
By \cite[Proposition~3.5]{FG},
\begin{equation}\label{FG}
\frac {q^n}2\le |\Irr(G)| =|\Cl(G)|\le q^n.
\end{equation}

Denote by $\cP$ the set of all integer partitions $\lambda$ (including the empty partition~$\emptyset$) and
by $\cF$ the set of all non-constant monic irreducible polynomials 
$f(x)\in \mathbb F_q[x]$ with non-zero constant term.
We define the \emph{degree} of $\nu$ as follows:
\[ \deg(\nu) := \sum_{f\in  \cF} \deg(f)|\nu(f)|.\]
By Jordan decomposition, there is a natural bijection between conjugacy classes in $G$
and maps $\nu:\cF \to  \cP$ of degree $n$.
Green  \cite{Green} introduced the set $\cG$ of \emph{simplices}
and proved (Theorem~12) that $\Irr(G)$ has a parametrization by
maps $\nu: \cG\to \cP$ satisfying
\[\sum_{f\in  \cG} \deg(f)|\nu(f)|=n.\]
By fixing in a compatible way multiplicative generators of finite fields, he gave a degree-preserving bijection between
$\cF$ and $\cG$.  
We will ignore the distinction between $\cF$ and $\cG$ henceforward.
The same theorem of Green also gave a formula for the degree of the irreducible character $\chi$ associated to $\nu$.  It can be written
\begin{equation}
\label{dchi}
d_\chi = q^{N_\nu}\frac{\prod_{i=1}^n (q^i-1)}{\prod_{f\in \cF}\prod_{i=1}^{|\nu(f)|} (q^{h_{\nu(f),i}\deg(f)}-1)},
\end{equation}
where $N_\nu$ is a certain non-negative integer, and the $h_{\lambda,i}$ are the hook lengths of the partition~$\lambda$; in particular these are positive integers $\le |\lambda|$.

By the \emph{support of $\nu$}, which we denote $\supp\nu$, we mean the set of $f\in \cF$ such that $\nu(f)\neq \emptyset$.

\begin{lemma}
\label{deficiency}
Let $\gamma$ be defined as in Lemma~\ref{rare}, and let $N$ be a positive integer.  Then the number of degree $n$ functions $\nu\colon \cF\to \cP$ satisfying
$\deg(f) (|\nu(f)|-1) \ge N$ for some $f$ is less than $\frac {2N\gamma^N}{(1-\gamma)^2}q^n$.
\end{lemma}

\begin{proof}
It suffices to prove that for each $m$, the number of choices of $\nu$ of degree $n$ such that for some $f\in \cF$, $\deg(f) (|\nu(f)|-1) = m$ is less than $2m\gamma^mq^n$.
Since there are at most $m$ ways of expressing $m$ as $a(b-1)$ for positive integers $a$ and $b$, it suffices to prove that there are less than $2\gamma^m q^n$
such $\nu$ of degree $n$ for which $|\nu(f)| = b$ for some $f\in \cF$ of degree $a$.  Since there are fewer than $q^a$ elements of $\cF$ of degree $a$, it suffices to prove
that for given $f\in \cF$ of degree $a$, there are at most $2\gamma^m q^{n-a}$ possibilities for $\nu$ with $|\nu(f)| = b$.  For each partition $\lambda$ of $b$, the functions $\nu$ of degree $n$ with $\nu(f) = \lambda$ can be put into bijective correspondence with $\nu'$ of degree $n-ab$
with $\nu'(f)=\emptyset$.  By \eqref{FG}, the number of possibilities for $\nu'$ and therefore for $\nu$ is at most $q^{n-ab} = q^{n-m-a}$.   Summing over the possibilities for $\lambda$, 
which by Lemma~\ref{rare} number less than $2\gamma^m q^m$, we obtain less than $2\gamma^m q^{n-a}$ possibilities for $\nu$ with $|\nu(f)|=b$, as claimed.
\end{proof}

We define the \emph{deficiency} of a character of $G$ or of the associated $\nu\colon \cF\to\cP$ to be the maximum of $\deg(f) (|\nu(f)|-1)$ over all $f\in \cF$.  Together, Lemma~\ref{deficiency}
and \eqref{FG} imply that for all $\varepsilon > 0$ there exists an $N$ such that for all $n$ and $q$, the proportion of irreducible characters of $\GL(n,q)$ with deficiency $<N$ is at least $1-\varepsilon$.
\begin{lemma}
\label{ord-l}
Let $m$ be a positive integer and $\ell$ a prime such that $\ell m > n$ and $\ord_\ell P_m(q) = e>0$.  Let $\chi$ be a character whose deficiency is less than $m/2$.  Then
\begin{align*}
\ord_\ell d_\chi &= e \lfloor n/m\rfloor - e|\{f\in \supp\nu\mid \deg(f)\in m\Z\}|\\
&= \ord_\ell |G| - e|\{f\in \supp\nu\mid \deg(f)\in m\Z\}|.
\end{align*}
\end{lemma}

\begin{proof}
If $f$ is in the support of $\nu$ and $\deg(f)|\nu(f)| < m$, then $f$ does not contribute any factor of $\ell$ to the denominator of \eqref{dchi}.  So we need only consider the case $\deg(f)|\nu(f)| \ge m$,
in which case $\deg(f) (|\nu(f)|-1) \ge m/2$ if $|\nu(f)|\ge 2$.  Since the deficiency of $\chi$ is less than $m/2$, this is impossible, which means that all $f$ contributing factors of $\ell$ in \eqref{dchi} satisfy $\nu(f) = (1)$.
Moreover, by Lemma~\ref{big-factor}, $\ell$ divides $q^k-1$ if and only if $m$ divides $k$, in which case $\ord_\ell (q^k-1) = e$.  Thus, the factors in \eqref{dchi} contributing to $\ord_\ell$ are
$q^m-1,q^{2m}-1,\ldots,q^{\lfloor n/m\rfloor m}-1$, each of which contributes $e$, and $q^{\deg(f)}-1$ for each $f\in \supp\nu$ of degree divisible by $m$, again each contributing $e$.
\end{proof}

\begin{lemma}
\label{degree-m}
For any positive integer $m$, the number of $\nu\colon \cF\to \cP$ of degree $n$ for which there exist $f\in \cF$ of degree $m$ with $\nu(f)=(1)$ is less than $q^n/m$.
\end{lemma}

\begin{proof}
Any degree $m$ element of $\cF$ splits completely in $\F_{q^m}$, so there are less than $q^m/m$ such elements.  For each $f$, there is a bijective correspondence between $\nu$ of degree $n$ with  $\nu(f) = (1)$ and $\nu'$ of degree $n-m$ with $\nu'(f) = \emptyset$.  By \eqref{FG}, there are at most $q^{n-m}$ such $\nu'$, so the total number of $\nu$ is less than $q^n/m$.
\end{proof}

\begin{lemma}
\label{order-equality}
For all $\varepsilon > 0$, if $n$ is sufficiently large in terms of $\varepsilon$, $m$ is a sufficiently large positive integer, $\ell$ is a prime divisor of $P_m(q)$, and $\ell m > n$, then the probability is at least 
\[1-\frac{2+2\log n - 2\log m}m-\varepsilon\]
that a random element $\chi$ chosen uniformly from $\Irr(G)$ satisfies
\begin{equation}
\label{l-order}
\ord_\ell d_\chi = \ord_\ell |G|.
\end{equation}
\end{lemma}

\begin{proof}
Choose $N$ in Lemma~\ref{deficiency} such that $N\gamma^N < (1-\gamma)^2\varepsilon/4$.  By \eqref{FG}, the probability that $\chi$ has deficiency $\ge N$ is less than $\varepsilon$.
We assume $m>2N$, so with probability greater than $1-\varepsilon$, the deficiency of a random $\chi\in \Irr(G)$ is less than $m/2$.  By Lemma~\ref{ord-l}, this implies
\eqref{l-order}
provided that no element in the support of $\nu$ has degree a multiple of $m$.  If $f\in \supp \nu$ has degree $km$, 
then the deficiency condition on $\nu$ implies $\nu(f) = (1)$.
By Lemma~\ref{degree-m}, the probability that there exists an element in the support of $\nu$ of degree $km$ is less than $2/km$,
so the probability that there is an element in the support of $\nu$ with degree in $m\Z$ is less than
\[\sum_{k=1}^{\lfloor n/m\rfloor} \frac 2{km} < \frac{2+2\log n - 2\log m}m.\qedhere\]
\end{proof}

\begin{lemma}
\label{R-condition}
For all $\delta >0$, if $n$ is sufficiently large in terms of $\delta$, $m\ge \sqrt n$, and $\ell$ is any prime divisor of $P_m(q)$, then the probability of \eqref{l-order} is greater than $1-\delta/2$.
\end{lemma}

\begin{proof}
By part \eqref{b-f:mod} of Lemma~\ref{big-factor}, $\ell>m$, so $\ell m > n$.  Applying Lemma~\ref{order-equality} for $\varepsilon = \delta/4$, the claim holds if 
\[\frac{2+2\log n - 2\log m}m < \frac \delta 4.\]
For $n\ge 8$ and $m\ge \sqrt n$, the left-hand side is less than $2n^{-1/2}\log n$, which goes to zero as $n$ goes to $\infty$.
\end{proof}
\section{Proof of Lemma~\ref{Lemma:B}}\label{Sect:Proof:Lemma:B}

Let $\Fact f$ denote the total number of factors in the decomposition of ${f(x)\in \F_q[x]}$ into irreducibles.
For each $g\in \GL(n,q)$, let $p_g(x)$ denote the characteristic polynomial of $g$.

\begin{lemma}
\label{nearly-squarefree}
There exist constants $A$ and $B$ such that for all $m$, $n$, and $q$, at most $A n^B q^{-m}|\GL(n,q)|$ elements of $\GL(n,q)$ have a characteristic polynomial with a repeated
irreducible factor of degree $\ge m$.
\end{lemma}

\begin{proof}
By \cite[Proposition~3.3]{LS}, the number of elements of $\GL(n,q)$ with any given characteristic polynomial is at most $(A/8) n^B q^{n^2-n}$ for some absolute constants $A$ and $B$.  (Actually, the statement is proven only for ``classical'' groups, but the proof for $\GL(n,q)$ is identical.)  For any given $f$ of degree $m$, there are $q^{n-2m}$ polynomials of degree $\le n$ divisible by $f^2$, so there are less than $q^{n-m}$ polynomials of degree $n$ with a repeated irreducible factor of degree $m$ and less than $q^{n-m}+q^{n-m-1} + \cdots < 2q^{n-m}$ polynomials with a repeated irreducible factor of degree $\ge m$. On the other hand, by the same argument as \eqref{quarter},
\[|\GL(n,q)| = \prod_{i=1}^n (q^n-q^i) > \frac {q^{n^2}}4.\]
The lemma follows.
\end{proof}

\begin{proof}[Proof of Lemma~\ref{Lemma:B}]

By \cite[Proposition~3.4]{LS}, for all $\delta > 0$ there exists $k$ such that 
\begin{equation}
\label{Shalev}
\Prob[\Fact p_g > k\log n] < \frac \delta4,
\end{equation}
where $\Prob$ denotes  probability with respect to the uniform distribution on $G=\GL(n,q)$.
(Actually, the cited reference proves the analogous claim for $\SL(n,q)$, but the proof goes through the $\GL(n,q)$ case.)
Choose $k$ so that this holds and assume that $n$ is large enough that 
\begin{enumerate}[\rm(a)]
\item\label{c:sqn} $\sqrt n > k\log n$,
\item\label{c:AnB} $A n^B 2^{-\sqrt n} < \frac \delta 4$, where $A$ and $B$ are defined as in Lemma~\ref{nearly-squarefree},
\item\label{c:sqm} $\sqrt{m/2} > \log_2 m + 2$ for all $m\ge \sqrt n$,
\item\label{c:m} $m > 1/\varepsilon$ for all $m\ge \sqrt n$.
\end{enumerate}

Let $\cX$ denote the set of elements $g$ for which $p_g(x)$ has $\le k\log n$ 
irreducible factors and no repeated factor of degree $\ge \sqrt n$.  By condition \eqref{c:sqn} on $n$, 
every $p_g$ with $g\in\cX$ has a simple irreducible factor of degree $\ge \sqrt n$.
By equation \eqref{Shalev} and condition \eqref{c:AnB}, $|G\setminus \cX| < (\delta/2)|G|$.
For each $g\in \cX$, fix an irreducible factor of degree $m_g \ge\sqrt n$ of $p_g$. 
By condition \eqref{c:sqm} and part \eqref{b-f:bound} of Lemma~\ref{big-factor}, $P_{m_g}(q) > 1$, so for each $g$, we may fix a prime divisor $\ell_g$ of $P_{m_g}(q)$.  We define $\cR$ to consist of all pairs $(\chi,g)$ where $g\not\in \cX$
or where $g\in \cX$ but 
\[\ord_{\ell_g} d_\chi \neq \ord_{\ell_g} |G|.\]
By Lemma~\ref{R-condition}, for each $g\in \cX$, there are at most $(\delta/2) |\Irr(G)|$ pairs ${(\chi,g)\in \cR}$.  Thus, $\cR$ satisfies equation \eqref{Lemma:B:second}.

For pairs $(\chi,g)\not\in \cR$, we have $g\in \cX$ and $\ord_{\ell_g} d_\chi = \ord_{\ell_g} |G|$.
As $p_g(x)$ has an irreducible factor of degree $m_g$ which occurs with multiplicity $1$, 
the centralizer of $g$ has order divisible by $q^{m_g}-1$ and therefore by $\ell_g$.  Therefore, $\ord_{\ell_g} s_g < \ord_{\ell_g} |G|$.
This implies that $\ell_g$ is a divisor of the denominator of $(d_\chi,s_g)/d_\chi$.  
As $\ell_g\equiv 1\pmod{m_g}$, we have $\ell_g > m_g$.  By condition \eqref{c:m} on $n$,  $m_g\ge 1/\varepsilon$.
Thus, equation \eqref{Lemma:B:first} holds.
\end{proof}


\begin{thebibliography}{10}
\bibitem{ErdosLehner}
P.~Erd\H{o}s and J.~Lehner,
The distribution of the number of summands in the partitions of a positive 
integer.
\emph{Duke Math. J.} {\bf 8} (1941) 335--345.

\bibitem{FG}
J.~Fulman and R.~Guralnick,
Bounds on the number and sizes of conjugacy classes in finite 
Chevalley groups with applications to derangements.
\emph{Trans.~Amer.~Math.~Soc.} {\bf 364} (2012) 3023--3070.

\bibitem{PXGallagher2}
P.~X.~Gallagher,
Degrees, class sizes and divisors of character values. 
\emph{J.~Group Theory} {\bf 15} (2012) 455--467.

\bibitem{Goncharoff}
V.~L.~Goncharoff,
Sur la distribution des cycles dans les permutations. 
C.~R.~(Doklady) Acad.~Sci.~URSS (N.S.) {\bf 35} (1942) 267--269.

\bibitem{Green}
J.~A.~Green,
The characters of the finite general linear groups. 
\emph{Trans.~Amer.~Math.~Soc.} {\bf 80} (1955) 402--447.

\bibitem{LS}
M.~Larsen and A.~Shalev,
On the distribution of values of certain word maps. 
\emph{Trans.~Amer.~Math.~Soc.} {\bf 368} (2016), no.\ 3, 1647--1661.

\bibitem{ARMiller1}
A.~R.~Miller,
The probability that a character value is zero for the symmetric group.
\emph{Math.~Z.} {\bf 277} (2014) 1011--1015.
\end{thebibliography}
\end{document}